\documentclass[12pt, reqno]{amsart}
\usepackage{amsmath, amsthm, amscd, amsfonts, amssymb, graphicx, color}
\usepackage[bookmarksnumbered, colorlinks, plainpages]{hyperref}

\input{mathrsfs.sty}

\newtheorem{theorem}{Theorem}[section]
\newtheorem{lemma}[theorem]{Lemma}

\newtheorem{corollary}[theorem]{Corollary}
\theoremstyle{definition}

\theoremstyle{remark}
\newtheorem{remark}[theorem]{Remark}
\numberwithin{equation}{section}
\begin{document}

\title [Some generalizations of the Aluthge transform of operators  ]{Some generalizations of the Aluthge transform of operators and their consequences}

\author[ M. Bakherad and  K. Shebrawi   ]{Mojtaba Bakherad$^1$ and Khalid Shebrawi$^2$}

\address{ $^1$Department of Mathematics, Faculty of Mathematics, University of Sistan and Baluchestan, Zahedan, I.R.Iran.}
\email{mojtaba.bakherad@yahoo.com; bakherad@member.ams.org}

\address{ $^2$Department of Mathematics, Al-Balqa' Applied University, Salt, Jordan.}
\email{khalid@bau.edu.jo; shebrawi@gmail.com}


%~~~~~~~~~~~~~~~~~~~~~~~~~~~~~~~~~~~~~~~~~~~~~~~~~~~~~~~~~~~~~~~~~~~~~~~~~~~~~~~~~~~~~~~~~~~~~~~~~~~~~~~~~~~~~~~~~~~~~~~~~~~~~~~~~~~
\begin{abstract}

Let $A = U |A|$  be the polar decomposition of $A$. The
Aluthge transform of the operator $A$, denoted by $\tilde{A}$, is defined as
$\tilde{A} =|A|^{\frac{1}{2}} U |A|^{\frac{1}{2}}$.
In this paper, first we generalize the definition of Aluthge transform for
non-negative continuous functions $f, g$ such that $f(x)g(x)=x\,\,(x\geq0)$. Then, by using of this definition, we get some numerical radius inequalities. Among other inequalities, it is shown that if $A$ is bounded linear operator on a complex Hilbert space ${\mathscr H}$, then
 \begin{equation*}
h\left( w(A)\right) \leq \frac{1}{4}\left\Vert h\left( g^{2}\left( \left\vert
A\right\vert \right) \right) +h\left( f^{2}\left( \left\vert A\right\vert
\right) \right) \right\Vert +\frac{1}{2}h\left( w\left( \tilde{A}_{f,g}\right)
\right) ,
\end{equation*}
where $f, g$ are non-negative continuous functions such that $f(x)g(x)=x\,\,(x\geq 0)$,  $h$ is a non-negative non-decreasing convex function on $[0,\infty )$ and $\tilde{A}_{f,g} =f(|A|) U g(|A|)$.
\end{abstract} \maketitle
%~~~~~~~~~~~~~~~~~~~~~~~~~~~~~~~~~~~~~~~~~~~~~~~~~~~~~~~~~~~~~~~~~~~~~~~~~~~~~~~~~~~~~~~~~~~~~~~~~~~~~~~~~~~~~~~~~~~~~~~~~~~~~~~~~~~
\section{Introduction}
Let ${\mathbb B}(\mathscr H)$ denote the $C^{*}$-algebra of all bounded linear operators on a complex Hilbert space ${\mathscr H}$ with an inner product $\langle \cdot,\cdot\rangle$ and the corresponding norm $ \Vert \cdot \Vert $. In the case when ${\rm dim}{\mathscr H}=n$, we identify ${\mathbb B}({\mathscr H})$ with the matrix algebra $\mathbb{M}_n$ of all $n\times n$ matrices with entries in
the complex field.
For an operator $A\in{\mathbb B}(\mathscr H)$, let $A = U |A|$ (U is a partial
isometry with $\textrm{ker} U = \textrm{rng} |A|^\bot$) be the polar decomposition of $A$. The
Aluthge transform of the operator $A$, denoted by $\tilde{A}$, is defined as
$\tilde{A} =|A|^{\frac{1}{2}} U |A|^{\frac{1}{2}}.$
 In \cite{okobo}, Okubo introduced a more general notion called $t$-Aluthge transform which has later
been studied also in detail. This is defined for any $0<t\leq1$ by
$\tilde{A}_t =|A|^{t} U |A|^{1-t}$. Clearly, for $t=\frac{1}{2}$ we obtain the usual Aluthge transform. As for the case $t=1$, the operator $\tilde{A}_1=|A|U$  is called the Duggal transform of $A\in{\mathbb B}(\mathscr H)$. For  $A\in{\mathbb B}(\mathscr H)$, we generalize the Aluthge transform of the operator $A$ to the form
\begin{align*}
\tilde{A}_{f,g} =f(|A|) U g(|A|),
\end{align*}
in which $f, g$ are non-negative continuous functions such that $f(x)g(x)=x\,\,(x\geq0).$ The numerical radius of $A\in {\mathbb B}({\mathscr H})$ is defined by
\begin{align*}
w(A):=\sup\{| \langle Ax, x\rangle| : x\in {\mathscr H}, \Vert x \Vert=1\}.
\end{align*}
It is well known that $w(\,\cdot\,)$ defines a norm on ${\mathbb B}({\mathscr H})$, which is equivalent to the usual operator norm $\Vert \cdot \Vert$. In fact, for any $A\in {\mathbb B}({\mathscr H})$,
$
\frac{1}{2}\Vert A \Vert\leq w(A) \leq\Vert A \Vert$;
 see \cite{gof}. Let $r (\cdot)$ denote to the spectral radius. It is well known that for every operator
$A\in{\mathbb B}({\mathscr H})$, we have $r(A)\leq w(A)$. An important inequality for $\omega(A)$ is the power inequality stating that $\omega(A^n)\leq \omega(A)^n\,\,(n=1,2,\cdots)$. The quantity $w(A)$ is useful in studying perturbation, convergence and approximation problems as well as integrative method, etc. For more information see \cite{Pol, hal, HKS2, HKS1} and references therein.\\
 Let $A, B, C, D\in {\mathbb B}({\mathscr H})$. The operator matrices  $\left[\begin{array}{cc} A&0\\ 0&D \end{array}\right]$  and $\left[\begin{array}{cc}  0&B\\ C&0 \end{array}\right]$ are called the diagonal and off-diagonal parts of the operator matrix $\left[\begin{array}{cc} A&B\\ C&D \end{array}\right]$, respectively.\\
In \cite{KIT1}, It has been shown that if $A$ is an operator in ${\mathbb B}({\mathscr H})$, then
\begin{eqnarray}\label{abcd}
w(A)\leq \frac{1}{2}\left(\|A\|+\|A^2\|^\frac{1}{2}\right).
\end{eqnarray}%
Several refinements and generalizations of inequality \eqref{abcd} have been given; see \cite{omer1,bakh,sheb,YAM}. Yamazaki \cite{YAM} showed that for  $A\in{\mathbb B}({\mathscr H})$ and  $t\in[0,1]$ we have
\begin{eqnarray}\label{yam-sheb}
w(A)\leq \frac{1}{2}\left(\|A\|+w(\tilde{A}_t)\right).
\end{eqnarray}%
 Davidson and Power \cite{dav} proved that if $A$ and $B$ are positive operators in ${\mathbb B}({\mathscr H})$, then
\begin{eqnarray}\label{abcd12}
\Vert A+B\Vert\leq \max\{\Vert A\Vert,\Vert B\Vert \}+\Vert AB\Vert^\frac{1}{2}.
\end{eqnarray}%
Inequality \eqref{abcd12} has been generalized in \cite{omer2,sheb1}. In \cite{sheb1}, the author extended this inequality to the form
\begin{eqnarray}\label{abcd123}
\Vert A+B^*\Vert\leq \max\{\Vert A\Vert,\Vert B\Vert \}+\frac{1}{2}\left(\left\Vert|A|^t|B^*|^{1-t}\right\Vert+ \left\Vert|A^*|^{1-t}|B|^t\right\Vert\right),
\end{eqnarray}%
in which $A, B\in{\mathbb B}({\mathscr H})$ and $t\in[0,1]$.

In this paper, by applying the generalized Aluthge transform of operators, we establish some inequalities involving the numerical radius. In particular, we  extend inequality \eqref{yam-sheb} and \eqref{abcd123} for two non-negative continuous functions. We also show some upper bounds for the numerical radius of $2\times2$ operators matrices.

 %===================================================================================================================================

\section{main results}
To prove our numerical radius inequalities, we need several known lemmas.
\begin{lemma}\cite[Theorem 2.2]{omer1}\label{kit1}
Let $X, Y, S, T\in {\mathbb B}({\mathscr H})$. Then
\begin{align*}
r(XY+ST)\leq\frac{1}{2}\left(w(YX)+w(TS)\right)+\frac{1}{2}\sqrt{\left(w(YX)-w(TS)\right)^2+4\Vert YS\Vert\Vert TX\Vert}.
\end{align*}
\end{lemma}
\begin{lemma}\cite{YAM,KIT1}\label{1}
Let $A\in {\mathbb B}({\mathscr H})$. Then\newline

$(a)\,\,w(A)=\underset{\theta \in
%TCIMACRO{\U{211d} }%
%BeginExpansion
\mathbb{R}
%EndExpansion
}{\max }\left\Vert \textrm{Re}\left( e^{i\theta }A\right) \right\Vert. $\newline

 $(b)\,\,w\left(\left[\begin{array}{cc}
              0&A\\
              0&0
              \end{array}\right]\right)
              = \frac{1}{2}\Vert A\Vert.$
\end{lemma}
\textbf{Polarization identity}: For all $x,y\in {\mathscr H}$, we have%
\begin{equation*}
\left\langle x,y\right\rangle =\frac{1}{4}\sum_{k=0}^{3}\left\Vert
x+i^{k}y\right\Vert ^{2}i^{k}.
\end{equation*}
Now, we are ready to present our first result. The following theorem shows a generalization of inequality \eqref{yam-sheb}.
\begin{theorem}\label{main-man}
Let $A\in {\mathbb{B}}({\mathscr H})$ and $f,g$ be two non-negative
continuous functions on $[0,\infty )$ such that $f(x)g(x)=x\,\,(x\geq 0)$.
Then, for all non-negative non-decreasing convex function $h$ on $[0,\infty )$, we have%
\begin{equation*}
h\left( w(A)\right) \leq \frac{1}{4}\left\Vert h\left( g^{2}\left( \left\vert
A\right\vert \right) \right) +h\left( f^{2}\left( \left\vert A\right\vert
\right) \right) \right\Vert +\frac{1}{2}h\left( w\left( \tilde{A}_{f,g}\right)
\right) .
\end{equation*}
\end{theorem}

\begin{proof}
Let $x$ be any unit vector. Then%
\begin{eqnarray*}
\textrm{Re}\left\langle e^{i\theta }Ax,x\right\rangle &=&\textrm{Re}\left\langle
e^{i\theta }U\left\vert A\right\vert x,x\right\rangle \\
&=&\textrm{Re}\left\langle e^{i\theta }Ug\left( \left\vert A\right\vert
\right) f\left( \left\vert A\right\vert \right) x,x\right\rangle \\
&=&\textrm{Re}\left\langle e^{i\theta }f\left( \left\vert A\right\vert \right)
x,g\left( \left\vert A\right\vert \right) U^{\ast }x\right\rangle \\
&=&\frac{1}{4}\left\Vert \left( e^{i\theta }f\left( \left\vert A\right\vert
\right) +g\left( \left\vert A\right\vert \right) U^{\ast }\right) x\right\Vert
^{2}-\frac{1}{4}\left\Vert \left( e^{i\theta }f\left( \left\vert A\right\vert
\right) -g\left( \left\vert A\right\vert \right) U^{\ast }\right) x\right\Vert ^{2}
\\
&&\qquad \qquad \qquad \qquad \qquad \qquad \qquad \qquad \text{\ (by
polarization identity)} \\
&\leq &\frac{1}{4}\left\Vert \left( e^{i\theta }f\left( \left\vert A\right\vert
\right) +g\left( \left\vert A\right\vert \right) U^{\ast }\right) x\right\Vert ^{2}
\\
&\leq &\frac{1}{4}\left\Vert \left( e^{i\theta }f\left( \left\vert A\right\vert
\right) +g\left( \left\vert A\right\vert \right) U^{\ast }\right) \right\Vert ^{2}
\\
&=&\frac{1}{4}\left\Vert \left( e^{i\theta }f\left( \left\vert A\right\vert
\right) +g\left( \left\vert A\right\vert \right) U^{\ast }\right) \left(
e^{-i\theta }f\left( \left\vert A\right\vert \right) +Ug\left( \left\vert
A\right\vert \right) \right) \right\Vert \\
&=&\frac{1}{4}\left\Vert g^{2}\left( \left\vert A\right\vert \right) +f^{2}\left(
\left\vert A\right\vert \right) +e^{i\theta }\tilde{A}_{f,g}+e^{-i\theta
}\left( \tilde{A}_{f,g}\right) ^{\ast }\right\Vert \\
&\leq &\frac{1}{4}\left\Vert g^{2}\left( \left\vert A\right\vert \right)
+f^{2}\left( \left\vert A\right\vert \right) \right\Vert +\frac{1}{4}\left\Vert
e^{i\theta }\tilde{A}_{f,g}+e^{-i\theta }\left( \tilde{A}_{f,g}\right)
^{\ast }\right\Vert \\
&=&\frac{1}{4}\left\Vert g^{2}\left( \left\vert A\right\vert \right) +f^{2}\left(
\left\vert A\right\vert \right) \right\Vert +\frac{1}{2}\left\Vert \textrm{Re}\left(
e^{i\theta }\tilde{A}_{f,g}\right) \right\Vert \\
&\leq &\frac{1}{4}\left\Vert g^{2}\left( \left\vert A\right\vert \right)
+f^{2}\left( \left\vert A\right\vert \right) \right\Vert +\frac{1}{2}w\left(
\tilde{A}_{f,g}\right) .
\end{eqnarray*}%
Now,  taking the supremum over all unit vector  $x\in{\mathscr H}$ and applying Lemma \ref{1} in the above inequality produces
\begin{eqnarray*}
w\left( A\right)
\leq \frac{1}{4}\left\Vert g^{2}\left( \left\vert A\right\vert \right)
+f^{2}\left( \left\vert A\right\vert \right) \right\Vert +\frac{1}{2}w\left(
\tilde{A}_{f,g}\right) .
\end{eqnarray*}%
Therefore,%
\begin{eqnarray*}
h\left( w\left( A\right) \right) &\leq &h\left( \frac{1}{4}\left\Vert g^{2}\left(
\left\vert A\right\vert \right) +f^{2}\left( \left\vert A\right\vert \right)
\right\Vert +\frac{1}{2}w\left( \tilde{A}_{f,g}\right) \right) \\
&=&h\left( \frac{1}{2}\left\Vert \frac{g^{2}\left( \left\vert A\right\vert
\right) +f^{2}\left( \left\vert A\right\vert \right) }{2}\right\Vert +\frac{1}{2}%
w\left( \tilde{A}_{f,g}\right) \right) \\
&\leq &\frac{1}{2}h\left( \left\Vert \frac{g^{2}\left( \left\vert A\right\vert
\right) +f^{2}\left( \left\vert A\right\vert \right) }{2}\right\Vert \right) +%
\frac{1}{2}h\left( w\left( \tilde{A}_{f,g}\right) \right)  \\&&\qquad\qquad\qquad\qquad\qquad\qquad(\textrm {by the convexity of }\, h )\\
&=&\frac{1}{2}\left\Vert h\left( \frac{g^{2}\left( \left\vert A\right\vert
\right) +f^{2}\left( \left\vert A\right\vert \right) }{2}\right) \right\Vert +%
\frac{1}{2}h\left( w\left( \tilde{A}_{f,g}\right) \right) \\&&\qquad\qquad\qquad\qquad\qquad\qquad(\textrm {by the functional calculus} )\\
&\leq &\frac{1}{4}\left\Vert h\left( g^{2}\left( \left\vert A\right\vert \right)
\right) +h\left( f^{2}\left( \left\vert A\right\vert \right) \right) \right\Vert +%
\frac{1}{2}h\left( w\left( \tilde{A}_{f,g}\right) \right)
\\&&\qquad\qquad\qquad\qquad\qquad\qquad(\textrm {by the convexity of }\, h ).
\end{eqnarray*}
\end{proof}
Theorem \ref{main-man} includes some special cases as follows.
\begin{corollary}
Let $A\in {\mathbb{B}}({\mathscr H})$. Then, for all non-negative non-decreasing convex
function $h$ on $[0,\infty )$ and all $t\in \left[ 0,1\right] $, we have%
\begin{equation}\label{soso}
h\left( w(A)\right) \leq \frac{1}{4}\left\Vert h\left( \left\vert A\right\vert
^{2t}\right) +h\left( \left\vert A\right\vert ^{2\left( 1-t\right) }\right)
\right\Vert +\frac{1}{2}h\left( w\left( \tilde{A}_{t}\right) \right) .
\end{equation}
\end{corollary}

\begin{corollary}\label{vaw}
Let $A\in {\mathbb{B}}({\mathscr H})$. Then, for all $t\in \left[ 0,1\right]
$ and $r\geq 1$, we have%
\begin{equation*}
w^{r}(A)\leq \frac{1}{4}\Vert \left\vert A\right\vert ^{2tr}+\left\vert
A\right\vert ^{2\left( 1-t\right) r}\Vert +\frac{1}{2}w^{r}\left( \tilde{A}%
_{t}\right) .
\end{equation*}%
In particular,%
\begin{equation*}
w^{r}(A)\leq \frac{1}{2}\left( \Vert A\Vert ^{r}+w^{r}\left( \tilde{A}%
\right) \right) .
\end{equation*}
\end{corollary}

\begin{proof}
The first inequality  follows from
inequality \eqref{soso} for the function  $h\left( x\right) =x^{r}\,\,(r\geq1)$. For the particular case, it is enough to put $t=\frac{1}{2}.$
\end{proof}
Theorem \ref{main-man} gives the next result for the off-diagonal operator matrix $\left[
\begin{array}{cc}
0 & A \\
B & 0%
\end{array}%
\right] $.
\begin{theorem}\label{man-kal}
Let $A, B\in {\mathbb B}({\mathscr H})$, $f,g$ be two non-negative
continuous functions on $[0,\infty )$ such that $f(x)g(x)=x\,\,(x\geq 0)$ and $r\geq1$. Then
\begin{eqnarray*}
w^{r}\left( \left[
\begin{array}{cc}
0 & A \\
B & 0%
\end{array}%
\right] \right) &\leq& \frac{1}{4}\max \left( \left\Vert  g^{2r}\left( \left\vert
A\right\vert  \right) + f^{2r}\left( \left\vert A\right\vert
\right)  \right\Vert ,\left\Vert  g^{2r}\left( \left\vert
B\right\vert  \right) + f^{2r}\left( \left\vert B\right\vert
\right)  \right\Vert\right) \\&&+\frac{1}{4}\big(\Vert f(|B|)g(|A^\ast|)\Vert^r+\Vert f(|A|)g(|B^\ast|)\Vert^r\big).
\end{eqnarray*}
\end{theorem}
\begin{proof}
Let $A=U|A|$ and $B=V|B|$ be the polar decompositions of $A$ and $B$, respectively and let $T=\left[
\begin{array}{cc}
0 & A \\
B & 0%
\end{array}%
\right]$. It follows from the polar the composition of $T=\left[
\begin{array}{cc}
0 & U \\
V & 0%
\end{array}%
\right]\left[
\begin{array}{cc}
|B| & 0 \\
0 & |A|%
\end{array}%
\right]$ that
\begin{eqnarray*}
\tilde{T}_{f,g}&=&f(|T|)\left[
\begin{array}{cc}
0 & U \\
V & 0%
\end{array}%
\right]g(|T|)\\&=&\left[
\begin{array}{cc}
f(|B|) & 0 \\
0 & f(|A|)%
\end{array}%
\right]\left[
\begin{array}{cc}
0 & U \\
V & 0%
\end{array}%
\right]\left[
\begin{array}{cc}
g(|B|) & 0 \\
0 & g(|A|)%
\end{array}%
\right]\\&
=&\left[
\begin{array}{cc}
0 & f(|B|)Ug(|A|) \\
f(|A|)Vg(|B|) & 0%
\end{array}%
\right].
\end{eqnarray*}
Using $|A^\ast|^2=AA^\ast=U|A|^2U^\ast$ and $|B^\ast|^2=BB^\ast=V|B|^2V^\ast$ we have $g(|A|)=U^\ast g(|A^\ast|)U$ and $g(|B|)=V^\ast g(|B^\ast|)V$ for every non-negative continuous function $g$ on $[0,\infty)$.
Therefore,
\begin{eqnarray}\label{do}
w\left(\tilde{T}_{f,g}\right)\nonumber
&=&w\left(\left[
\begin{array}{cc}
0 & f(|B|)Ug(|A|) \\
f(|A|)Vg(|B|) & 0%
\end{array}%
\right]\right)\nonumber\\&\leq& w\left(\left[
\begin{array}{cc}
0 & f(|B|)Ug(|A|) \\
0 & 0%
\end{array}%
\right]\right)+w\left(\left[
\begin{array}{cc}
0 & 0 \\
f(|A|)Vg(|B|) & 0%
\end{array}%
\right]\right)\nonumber\\&=& w\left(\left[
\begin{array}{cc}
0 & f(|B|)Ug(|A|) \\
0 & 0%
\end{array}%
\right]\right)+w\left(U^*\left[
\begin{array}{cc}
0 & f(|A|)Vg(|B|) \\
0 & 0%
\end{array}%
\right]U\right)\nonumber\\&=& w\left(\left[
\begin{array}{cc}
0 & f(|B|)Ug(|A|) \\
0 & 0%
\end{array}%
\right]\right)+w\left(\left[
\begin{array}{cc}
0 & f(|A|)Vg(|B|) \\
0 & 0%
\end{array}%
\right]\right)\nonumber\\
&= &\frac{1}{2}\Vert f(|B|)U g(|A|)\Vert+\frac{1}{2}\Vert f(|A|)V g(|B|)\Vert\nonumber\\&&
\qquad\qquad\qquad\qquad\qquad\textrm{(by Lemma \ref{kit1}(b))}\nonumber\\
&=&\frac{1}{2}\Vert f(|B|)U U^\ast g(|A^\ast|)U\Vert+\frac{1}{2}\Vert f(|A|)V V^\ast g(|B^\ast|)V\Vert\nonumber\\
&\leq&\frac{1}{2}\Vert f(|B|)g(|A^\ast|)\Vert+\frac{1}{2}\Vert f(|A|)g(|B^\ast|)\Vert,
\end{eqnarray}
where  $U=\left[
\begin{array}{cc}
0 & I \\
I & 0%
\end{array}%
\right]$ is unitary. Applying Theorem \ref{main-man} and inequality \eqref{do}, we have
\begin{eqnarray*}
w^{r}\left(T \right) &\leq& \frac{1}{4}\left\Vert  g^{2r}\left( \left\vert
T\right\vert  \right) + f^{2r}\left( \left\vert T\right\vert
\right)  \right\Vert +\frac{1}{2}\left( w^r\left( \tilde{T}_{f,g}\right)
\right) \\&\leq& \frac{1}{4}\max \left( \left\Vert  g^{2r}\left( \left\vert
A\right\vert  \right) + f^{2r}\left( \left\vert A\right\vert
\right)  \right\Vert ,\left\Vert  g^{2r}\left( \left\vert
B\right\vert  \right) + f^{2r}\left( \left\vert B\right\vert
\right)  \right\Vert\right) \\&&+\frac{1}{2}\left[\frac{1}{2}\left(\Vert f(|B|)g(|A^\ast|)\Vert+\Vert f(|A|)g(|B^\ast|)\Vert\right)\right]
^{r}\\&\leq& \frac{1}{4}\max \left( \left\Vert  g^{2r}\left( \left\vert
A\right\vert  \right) + f^{2r}\left( \left\vert A\right\vert
\right)  \right\Vert ,\left\Vert  g^{2r}\left( \left\vert
B\right\vert  \right) + f^{2r}\left( \left\vert B\right\vert
\right)  \right\Vert\right) \\&&+\frac{1}{4}\Vert f(|B|)g(|A^\ast|)\Vert^r+\frac{1}{4}\Vert f(|A|)g(|B^\ast|)\Vert^r\\&&\qquad\qquad
(\textrm{by the convexity}\,h(x)=x^r).
\end{eqnarray*}
\end{proof}
\begin{corollary}
\label{c1}Let $A, B\in {\mathbb B}({\mathscr H})$. Then,
for all $t\in \left[ 0,1\right] $ and $r\geq 1$, we have%
\begin{eqnarray*}
w^{\frac{r}{2}}\left( AB\right) &\leq& \frac{1}{4}\max \left( \left\Vert
\left\vert A\right\vert ^{2tr}+\left\vert A\right\vert ^{2\left( 1-t\right)
r}\right\Vert ,\left\Vert \left\vert B\right\vert ^{2tr}+\left\vert
B\right\vert ^{2\left( 1-t\right) r}\right\Vert \right) \\&&+\frac{1}{4}\left(
\left\Vert |A|^{t}\left\vert B^{\ast }\right\vert ^{1-t}\right\Vert^{r}
+\left\Vert \left\vert B\right\vert ^{t}|A^{\ast }|^{1-t}\right\Vert^{r} \right)
.
\end{eqnarray*}
\end{corollary}

\begin{proof}
Applying the power inequality of the numerical radius, we have%
\begin{eqnarray*}
w^{\frac{r}{2}}\left( AB\right) &\leq &\max \left( w^{\frac{r}{2}}\left(
AB\right) ,w^{\frac{r}{2}}\left( BA\right) \right) \\
&=&w^{\frac{r}{2}}\left( \left[
\begin{array}{cc}
AB & 0 \\
0 & BA%
\end{array}%
\right] \right) \\
&=&w^{\frac{r}{2}}\left( \left[
\begin{array}{cc}
0 & A \\
B & 0%
\end{array}%
\right] ^{2}\right) \\
&\leq &w^{r}\left( \left[
\begin{array}{cc}
0 & A \\
B & 0%
\end{array}%
\right] \right) \\
&\leq &\frac{1}{4}\max \left( \left\Vert \left\vert A\right\vert
^{2tr}+\left\vert A\right\vert ^{2\left( 1-t\right) r}\right\Vert
,\left\Vert \left\vert B\right\vert ^{2tr}+\left\vert B\right\vert ^{2\left(
1-t\right) r}\right\Vert \right) \\&&+\frac{1}{4}\left( \left\Vert
|A|^{t}\left\vert B^{\ast }\right\vert ^{1-t}\right\Vert^{r} +\left\Vert
\left\vert B\right\vert ^{t}|A^{\ast }|^{1-t}\right\Vert^{r} \right)\\&&
\qquad\qquad\qquad\qquad\textrm{(by Theorem \ref{man-kal})}.
\end{eqnarray*}
\end{proof}

\begin{corollary}
Let $A, B\in {\mathbb B}({\mathscr H})$ be positive operators. Then,
for all $t\in \left[ 0,1\right] $ and $r\geq 1$, we have%
\begin{eqnarray*}
\left\Vert A^{\frac{1}{2}}B^{\frac{1}{2}}\right\Vert ^{r}&\leq &\frac{1}{4}%
\max \left( \left\Vert  A^{tr}+
A^{\left( 1-t\right) r}\right\Vert ,\left\Vert
B^{tr}+ B^{2\left( 1-t\right)
r}\right\Vert \right) \\&&+\frac{1}{4}\left( \left\Vert A^{t}
B ^{1-t}\right\Vert ^{r} +\left\Vert B
^{t}A^{1-t}\right\Vert ^{r} \right).
\end{eqnarray*}
\end{corollary}

\begin{proof}
Since the spectral radius of any operator is dominated by its numerical
radius, then $r^{\frac{1}{2}}\left( AB\right) \leq w^{\frac{1}{2}}\left(
AB\right) .$ Applying a commutativity property of the spectral radius, we get%
\begin{eqnarray}
r^{\frac{r}{2}}\left( AB\right) &=&r^{\frac{r}{2}}\left( A^{\frac{1}{2}}A^{%
\frac{1}{2}}B^{\frac{1}{2}}B^{\frac{1}{2}}\right)  \notag \\
&=&r^{\frac{r}{2}}\left( A^{\frac{1}{2}}B^{\frac{1}{2}}B^{\frac{1}{2}}A^{%
\frac{1}{2}}\right)  \notag \\
&=&r^{\frac{r}{2}}\left( A^{\frac{1}{2}}B^{\frac{1}{2}}\left( A^{\frac{1}{2}%
}B^{\frac{1}{2}}\right) ^{\ast }\right)  \notag \\
&=&\left\Vert A^{\frac{1}{2}}B^{\frac{1}{2}}\left( A^{\frac{1}{2}}B^{\frac{1%
}{2}}\right) ^{\ast }\right\Vert ^{\frac{r}{2}}  \notag \\
&=&\left\Vert A^{\frac{1}{2}}B^{\frac{1}{2}}\right\Vert ^{r}.  \label{5}
\end{eqnarray}%
Now, the result follows from Corollary \ref{c1}.
\end{proof}
An important special case of Theorem \ref{man-kal}, which refines inequality \eqref{abcd123} can be stated as follows.
\begin{corollary}
Let $A, B\in {\mathbb B}({\mathscr H})$ and $r\geq1$. Then%
\begin{eqnarray*}
\left\Vert A+B\right\Vert ^{r}&\leq& \frac{1}{2^{2-r}}\max \left( \left\Vert
\left\vert A\right\vert ^{2tr}+\left\vert A\right\vert ^{2\left( 1-t\right)
r}\right\Vert ,\left\Vert \left\vert B^*\right\vert ^{2tr}+\left\vert
B^*\right\vert ^{2\left( 1-t\right) r}\right\Vert \right) \\&&+\frac{1}{2^{2-r}}\left( \left\Vert
|A|^{t}\left\vert B\right\vert ^{1-t}\right\Vert^{r} +\left\Vert
\left\vert B^*\right\vert ^{t}|A^{\ast }|^{1-t}\right\Vert^{r} \right) .
\end{eqnarray*}%
In particular, if $A$ and $B$ are normal, then
\begin{eqnarray*}
\left\Vert A+B\right\Vert ^{r}
\leq \frac{1}{2^{1-r}}\max \left( \left\Vert A\right\Vert ^{r},\left\Vert B\right\Vert
^{r}\right) +\frac{1}{2^{1-r}}\left\Vert AB\right\Vert ^{\frac{r}{2}}.
\end{eqnarray*}
\end{corollary}

\begin{proof}
Applying Lemma \ref{1} and Theorem \ref{main-man}, we have%
\begin{eqnarray*}
\left\Vert A+B^{\ast }\right\Vert ^{r} &=&\left\Vert T+T^{\ast }\right\Vert
^{r} \\
&\leq &2^r\underset{\theta \in
%TCIMACRO{\U{211d} }%
%BeginExpansion
\mathbb{R}
%EndExpansion
}{\max }\left\Vert \textrm{Re}\left( e^{i\theta }T\right) \right\Vert ^{r} \\
&=&2^rw^{r}\left( T\right) \\
&\leq &\frac{2^r}{4}\max \left( \left\Vert \left\vert A\right\vert
^{2tr}+\left\vert A\right\vert ^{2\left( 1-t\right) r}\right\Vert
,\left\Vert \left\vert B\right\vert ^{2tr}+\left\vert B\right\vert ^{2\left(
1-t\right) r}\right\Vert \right) \\&&+\frac{2^r}{4}\left( \left\Vert |A|^{t}\left\vert
B^{\ast }\right\vert ^{1-t}\right\Vert^{r} +\left\Vert \left\vert B\right\vert
^{t}|A^{\ast }|^{1-t}\right\Vert^{r} \right) \\&&
\textrm{(by Theorem \ref{man-kal})},
\end{eqnarray*}%
where $T=\left[
\begin{array}{cc}
0 & A \\
B & 0%
\end{array}%
\right] .$ Now, the desired result follows by replacing $B$ by $B^{\ast }$.
For the particular case, since $A$ and $B$ are normal,
then $\left\vert B^{\ast }\right\vert =$ $\left\vert B\right\vert $ and $%
|A^{\ast }|=|A|.$ Applying equality $($\ref{5}$)$ for the operators $%
\left\vert A\right\vert ^{\frac{1}{2}}$ and $\left\vert B\right\vert ^{\frac{%
1}{2}}$, we have
\begin{eqnarray*}
\left\Vert \left\vert A\right\vert ^{\frac{1}{2}}\left\vert B\right\vert ^{%
\frac{1}{2}}\right\Vert ^{r} &=&r^{\frac{r}{2}}\left( \left\vert
A\right\vert \left\vert B\right\vert \right) \\
&\leq &\left\Vert \left\vert A\right\vert \left\vert B\right\vert
\right\Vert ^{\frac{r}{2}} \\
&=&\left\Vert U^{\ast }AB^{\ast }V\right\Vert ^{\frac{r}{2}} \\
&=&\left\Vert AB^{\ast }\right\Vert ^{\frac{r}{2}},
\end{eqnarray*}%
where $A=U\vert A\vert$ and $B=V\left\vert B\right\vert $ are the
polar decompositions of the operators $A$ and $B.$ This completes the proof
of the corollary.
\end{proof}
In the next result, we show another generalization of inequality \eqref{yam-sheb}.
\begin{theorem}\label{main1}
Let $A\in {\mathbb B}({\mathscr H})$ and $f,g, h $ be  non-negative non-decreasing continuous functions on $[0,\infty)$ such that $f(x)g(x)=x\,\,(x\geq0)$. Then
\begin{align*}
h\left(w(A)\right)\leq \frac{1}{2}\Big(h\left(w\left(\tilde{A}_{f,g}\right)\right)
+ \left\Vert h(|A|)\right\Vert \Big).
\end{align*}
\end{theorem}
\begin{proof}
Let $A=U|A|$ be the polar decomposition of A.
Then for every $\theta\in\mathbb{R}$, we have
\begin{eqnarray}\label{eq1}
\Vert\textrm{Re}\left( e^{i\theta }A\right)\Vert&=&r\left(\textrm{Re}\left( e^{i\theta }A\right)\right)\nonumber\\&
=&\frac{1}{2}r\left(e^{i\theta }A+e^{-i\theta }A^\ast\right)\nonumber\\&
=&\frac{1}{2}r\left(e^{i\theta }U|A|+e^{-i\theta }|A|U^\ast\right)\nonumber\\&
=&\frac{1}{2}r\left(e^{i\theta }Ug(|A|)f(|A|)+e^{-i\theta }f(|A|)g(|A|)U^\ast\right).
\end{eqnarray}
Now, if we put $X=e^{i\theta}Ug(|A|)$, $Y=f(|A|)$, $S=e^{-i\theta}f(|A|)$ and $T=g(|A|)U^\ast$ in Lemma \ref{kit1}, then we get
\begin{eqnarray}\label{eq2}
\quad&&\hspace{-2.5 cm}r\big(e^{i\theta }Ug(|A|)f(|A|)+e^{-i\theta }f(|A|)g(|A|)U^\ast\big)\nonumber\\&\leq&
\frac{1}{2}\Big(w(f(|A|)Ug(|A|))+w(g(|A|)U^\ast f(|A|))\Big)\nonumber\\&&+\frac{1}{2}
\sqrt{4\Vert e^{-i\theta }f(|A|)g(|A|)\Vert\Vert g(|A|)U^\ast e^{i\theta }Uf(|A|)\Vert}\nonumber\\&&
\qquad\qquad\qquad\qquad\qquad\textrm{(by Lemma \ref{kit1})}\nonumber\\&
\leq &w(f(|A|)Ug(|A|))
+\sqrt{\Vert f(|A|)\Vert \Vert f(|A|)\Vert\Vert g(|A|)\Vert \Vert g(|A|)\Vert}\nonumber\\&
= &w(f(|A|)Ug(|A|))
+\sqrt{ f(\Vert A\Vert) g(\Vert A\Vert) g(\Vert A\Vert) f(\Vert A\Vert)}\nonumber\\&&
\qquad\qquad\qquad\qquad\qquad\textrm{(by the functional calculus)}\nonumber\\&=&w(f(|A|)Ug(|A|))
+\sqrt{ \Vert A\Vert \Vert A\Vert }\nonumber\\&=&w\left(\tilde{A}_{f,g}\right)
+ \Vert A\Vert.
\end{eqnarray}
Using inequalities \eqref{eq1}, \eqref{eq2} and Lemma \ref{1} we get
\begin{align*}
\omega(A)=\underset{\theta \in
%TCIMACRO{\U{211d} }%
%BeginExpansion
\mathbb{R}
%EndExpansion
}{\max }\left\Vert \textrm{Re}\left( e^{i\theta }A\right) \right\Vert\leq \frac{1}{2}\Big(w\left(\tilde{A}_{f,g}\right)
+ \Vert A\Vert\Big) .
\end{align*}
Hence
\begin{eqnarray*}
h\left(w(A)\right)&\leq& h\left(\frac{1}{2}\left[w\left(\tilde{A}_{f,g}\right)
+ \left\Vert A\right\Vert \right]\right)
\\&&\qquad\qquad\qquad(\textrm{by the monotonicity of}\,\,h)
\\&\leq& \frac{1}{2}h\left(w\left(\tilde{A}_{f,g}\right)\right)
+ \frac{1}{2}h\left(\left\Vert A\right\Vert \right)
\\&&\qquad\qquad\qquad(\textrm{by the convexity of}\,\,h)
\\&=&\frac{1}{2}h\left(w\left(\tilde{A}_{f,g}\right)\right)
+ \frac{1}{2}\left\Vert h(|A|)\right\Vert,
\end{eqnarray*}
as required.
\end{proof}
\textbf{Another proof for Theorem \ref{main-man}}:
We can obtain Theorem \ref{main-man} from Theorem \ref{main1}.
To see this, first note that by the hypotheses of Theorem \ref{main-man} we have
\begin{eqnarray}\label{cic}
h(|A|)&=&h(g(|A|)f(|A|))\nonumber\\&\leq &h\left(\frac{g^2(|A|)+f^2(|A|)}{2}\right)\qquad\textrm{(by the arithmetic-geometric inequality)}\nonumber\\&\leq&\frac{1}{2}\left(h\left(g^2(|A|)\right)+h\left(f^2(|A|)\right)\right)\qquad\textrm{(by the convexity of\,} h).
\end{eqnarray}
Hence, using Theorem \ref{main1} and inequality \eqref{cic} we get
\begin{eqnarray*}
h\left(w(A)\right)&\leq& \frac{1}{2}\Big[h\left(w\left(\tilde{A}_{f,g}\right)\right)
+ \left\Vert h(|A|)\right\Vert \Big]
\\&\leq& \frac{1}{2}\Big[h\left(w\left(\tilde{A}_{f,g}\right)\right)
+ \frac{1}{2}\left\Vert h\left(g^2(|A|)\right)+h\left(f^2(|A|)\right)\right\Vert \Big]
\\&=&\frac{1}{2}h\left(w\left(\tilde{A}_{f,g}\right)\right)
+ \frac{1}{4}\left\Vert h\left(g^2(|A|)\right)+h\left(f^2(|A|)\right)\right\Vert .
\end{eqnarray*}
\begin{remark}
For the special case $f(x)=x^t$ and $g=x^{1-t}\,\,(t\in[0,1])$, we obtain the inequality \eqref{yam-sheb}
\begin{align*}
w(A)\leq \frac{1}{2}\left(w\left(\tilde{A}_{t}\right)+\Vert A\Vert\right),
\end{align*}
where $A\in {\mathbb B}({\mathscr H})$ and $t\in[0,1]$.
\end{remark}
Using Theorem \ref{main1}, we get the following result.
\begin{corollary}
Let $A, B\in {\mathbb B}({\mathscr H})$   and $f,g$ be two non-negative non-decreasing continuous functions such that $f(x)g(x)=x\,\,(x\geq0)$. Then
\begin{align*}
2w^r\left(\left[
\begin{array}{cc}
0 & A \\
B & 0%
\end{array}%
\right]\right)\leq\max\{\Vert A\Vert^r,\Vert B\Vert^r\}+\frac{1}{2}\big(\left\Vert f(|B|)g(|A^\ast|)\right\Vert^r+\left\Vert f(|A|)g(|B^\ast|)\right\Vert^r\big),
\end{align*}
where $r\geq1$.
\end{corollary}
\begin{proof}
Using Theorem \ref{main1} and inequality \eqref{do}, we have
\begin{align*}
2w^r\left(\left[
\begin{array}{cc}
0 & A \\
B & 0%
\end{array}%
\right]\right)&\leq\left\Vert\left[
\begin{array}{cc}
0 & A \\
B & 0%
\end{array}%
\right]\right\Vert^r+w^r\left(\tilde{T}_{f,g}\right)&
\\&=
\max\{\Vert A\Vert^r,\Vert B\Vert^r\}+\left(\frac{1}{2}\big[\left\Vert f(|B|)g(|A^\ast|)\right\Vert+\Vert f(|A|)g(|B^\ast|)\Vert\big]\right)^r&
\\&\leq
\max\{\Vert A\Vert^r,\Vert B\Vert^r\}+\frac{1}{2}\big(\left\Vert f(|B|)g(|A^\ast|)\right\Vert^r+\Vert f(|A|)g(|B^\ast|)\Vert^r\big)
\end{align*}
and the proof is complete.
\end{proof}

\begin{corollary}
Let $A, B\in {\mathbb B}({\mathscr H})$  and $f,g$ be two non-negative non-decreasing continuous functions on $[0,\infty)$ such that $f(x)g(x)=x\,\,(x\geq0)$. Then
\begin{eqnarray*}
 \Vert A+B\Vert\leq\max\{\Vert A\Vert,\Vert B\Vert\}+ \frac{1}{2}\Big(\big\Vert f(|B|)g(|A|)\big\Vert+\big\Vert f(|A^\ast|)g(|B^\ast|)\big\Vert\Big).
\end{eqnarray*}
\end{corollary}
\begin{proof}
Let $T=\left[
\begin{array}{cc}
0 & A \\
B & 0%
\end{array}%
\right]$. Then
\begin{eqnarray*}
\left\Vert A+B^{^\ast }\right\Vert  &=&\left\Vert T+T^{^\ast }\right\Vert
 \\
&\leq &2\underset{\theta \in
%TCIMACRO{\U{211d} }%
%BeginExpansion
\mathbb{R}
%EndExpansion
}{\max }\left\Vert \textrm{Re}\left( e^{i\theta }T\right) \right\Vert  \\
&=&w(T)\qquad\qquad\qquad(\textrm {by Lemma \ref{1}})\\
&\leq&\max\{\Vert A\Vert,\Vert B\Vert\}+ \frac{1}{2}\Big(\big\Vert f(|B|)g(|A^\ast|)\big\Vert+\big\Vert f(|A|)g(|B^\ast|)\big\Vert\Big)\\
&&\qquad\qquad\qquad\qquad(\textrm {by Theorem \ref{main1}}).
\end{eqnarray*}
If we replace $B$ by  $B^\ast$, then we get the desired result.
\end{proof}
%##################################################
In the last results, we present some upper bounds for operator matrices. For this purpose, we need the following lemma.
\begin{lemma}\cite[Theorem 1]{KIT}\label{man5}
Let $A\in{\mathbb B}({\mathscr H})$ and $x, y\in {\mathscr H}$ be any vectors.  If $f$, $g$ are non-negative  continuous functions on $[0, \infty)$ which are satisfying the relation $f(x)g(x)=x\,(x\geq0)$,  then
\begin{align*}
| \left\langle Ax, y \right\rangle |^2 \leq \left\langle f^2(|A |)x ,x \right\rangle\, \left\langle g^2(| A^*|)y,y\right\rangle.
 \end{align*}
 \end{lemma}

 \begin{theorem}
 Let
 $
 A, B, C, D\in {\mathbb B}({\mathscr H})$ and $f_i, g_i\,\,(1\leq i\leq 4)$ be non-negative continues functions such that $f_i(x)g_i(x)=x\,\,(1\leq i\leq 4)$ for all $x\in[0,\infty)$. Then
\begin{align*}
\omega \left(\left[\begin{array}{cc}
 A&B\\
 C&D
 \end{array}\right]\right)&\leq\max\left\{\left\|f_1^2(|A|)+g_2^2(|B^*|)+f_3^2(|C|)\right\|^\frac{1}{2},\left\Vert g_4(|D^*|)\right\Vert\right\}
 \\&\,\,+\max\left\{\left\Vert g_1(|A^*|)\right\Vert,\left\|f_2^2(|B|)+g_3^2(|C^*|)+f_4^2(|D|)\right\|^\frac{1}{2}\right\}.
\end{align*}%
 \end{theorem}
 \begin{proof}
Let $T=\left[\begin{array}{cc}
 A&B\\
 C&D
 \end{array}\right]$ and $\mathbf{x}=\left[
\begin{array}{c}
x_1 \\
x_2%
\end{array}%
\right] $ be a unit vector (i.e., $\|x_1\|^2+\|x_2\|^2=1$). Then

{\footnotesize\begin{eqnarray*}
\left\vert \left\langle T\mathbf{x},\mathbf{x}\right\rangle \right\vert & =&\left\vert
\left\langle \left[
\begin{array}{cc}
A& B \\
C & D%
\end{array}%
\right] \left[
\begin{array}{c}
x_1 \\
x_2%
\end{array}%
\right] ,\left[
\begin{array}{c}
x_1 \\
x_2%
\end{array}%
\right] \right\rangle \right\vert \\
& =&\left\vert \left\langle \left[
\begin{array}{c}
Ax_1+Bx_2 \\
Cx_1+Dx_2%
\end{array}%
\right] ,\left[
\begin{array}{c}
x_1 \\
x_2%
\end{array}%
\right] \right\rangle \right\vert \\
& =&\left\vert \left\langle Ax_1,x_1\right\rangle +\left\langle
Bx_2,x_1\right\rangle +\left\langle Cx_1,x_2\right\rangle +\left\langle
Dx_2,x_2\right\rangle \right\vert \\
& \leq&\left\vert \left\langle Ax_1,x_1\right\rangle \right\vert +\left\vert
\left\langle Bx_2,x_1\right\rangle \right\vert +\left\vert \left\langle
Cx_1,x_2\right\rangle \right\vert +\left\vert \left\langle
Dx_2,x_2\right\rangle \right\vert\\&&
\left\langle f_1^2(|A|)x_1,x_1\right\rangle^\frac{1}{2}  \left\langle g_1^2(|A^*|)x_1,x_1\right\rangle^\frac{1}{2}+
 \left\langle f_2^2(|B|)x_2,x_2\right\rangle^\frac{1}{2}  \left\langle g_2^2(|B^*|)x_1,x_1\right\rangle^\frac{1}{2} \\&&+
 \left\langle f_3^2(|C|)x_1,x_1\right\rangle^\frac{1}{2}  \left\langle g_3^2(|C^*|)x_2,x_2\right\rangle^\frac{1}{2}+
 \left\langle f_4^2(|D|)x_2x_2\right\rangle^\frac{1}{2}  \left\langle g_4^2(|D^*|)x_2,x_2\right\rangle^\frac{1}{2}\\&
 \leq&\left(\left\langle f_1^2(|A|)x_1,x_1\right\rangle+\left\langle g_2^2(|B^*|)x_1,x_1\right\rangle+\left\langle f_3^2(|C|)x_1,x_1\right\rangle+\left\langle g_4^2(|D^*|)x_2,x_2\right\rangle\right)^\frac{1}{2}\\&&
 \qquad\qquad\qquad\qquad\qquad\textrm{(\normalsize{by the Cauchy-Schwarz inequality})}\\&&+
 \left(\left\langle g_1^2(|A^*|)x_1,x_1\right\rangle+\left\langle f_2^2(|B|)x_2,x_2\right\rangle+\left\langle g_3^2(|C^*|)x_2,x_2\right\rangle+\left\langle f_4^2(|D|)x_1,x_1\right\rangle\right)^\frac{1}{2}\\&=&
 \left(\left\langle \left(f_1^2(|A|)+g_2^2(|B^*|)+f_3^2(|C|)\right)x_1,x_1\right\rangle+\left\langle g_2^2(|B^*|)x_2,x_2\right\rangle\right)^\frac{1}{2}\\&&+
 \left(\left\langle \left(f_2^2(|B|)+g_3^2(|C^*|)+f_4^2(|D|)\right)x_2,x_2\right\rangle+\left\langle g_1^2(|A^*|)x_1,x_1\right\rangle\right)^\frac{1}{2}
 \\&\leq&
 \left(\left\| f_1^2(|A|)+g_2^2(|B^*|)+f_3^2(|C|)\right\|\|x_1\|^2+\left\| g_4^2(|D^*|)\right\|\|x_2\|^2\right)^\frac{1}{2}\\&&+
 \left(\left\| f_2^2(|B|)+g_3^2(|C^*|)+f_4^2(|D|)\right\|\|x_2\|^2+\left\| g_1^2(|A^*|)\right\|\|x_1\|^2\right)^\frac{1}{2}.
\end{eqnarray*}}%
Let
\begin{align*}&\alpha=\left\| f_1^2(|A|)+g_2^2(|B^*|)+f_3^2(|C|)\right\|,\quad
\beta=\left\| g_4^2(|D^*|)\right\|,\\&
\mu=\left\| f_2^2(|B|)+g_3^2(|C^*|)+f_4^2(|D|)\right\|
\quad\textrm{and} \,\,
\lambda=\left\| g_1^2(|A^*|)\right\|.
\end{align*} It follows from
$$\underset{%TCIMACRO{\U{211d} }%%BeginExpansion
  \|x_1\|^2+\|x_2\|^2=1
  %EndExpansion
}{\max}\left(\alpha\|x_1\|^2+\beta\|x_2\|^2\right)=\underset{\theta \in%TCIMACRO{\U{211d} }%%BeginExpansion
  [0,2\pi]
  %EndExpansion
}{\max}(\alpha\sin^2\theta+\beta\cos^2\theta)=\max\{\alpha,\beta\}$$
and
$$\underset{%TCIMACRO{\U{211d} }%%BeginExpansion
  \|x_1\|^2+\|x_2\|^2=1
  %EndExpansion
}{\max}\left(\lambda\|x_1\|^2+\mu\|x_2\|^2\right)=\underset{\theta \in%TCIMACRO{\U{211d} }%%BeginExpansion
  [0,2\pi]
  %EndExpansion
}{\max}(\lambda\sin^2\theta+\mu\cos^2\theta)=\max\{\lambda,\mu\}$$
that
\begin{eqnarray*}
\left\vert \left\langle T\mathbf{x},\mathbf{x}\right\rangle \right\vert&\leq&
 \left(\left\| f_1^2(|A|)+g_2^2(|B^*|)+f_3^2(|C|)\right\|\|x_1\|^2+\left\| g_4^2(|D^*|)\right\|\|x_2\|^2\right)^\frac{1}{2}\\&&+
 \left(\left\| f_2^2(|B|)+g_3^2(|C^*|)+f_4^2(|D|)\right\|\|x_2\|^2+\left\| g_1^2(|A^*|)\right\|\|x_1\|^2\right)^\frac{1}{2}\\&\leq&
 \max\left\{\left\|f_1^2(|A|)+g_2^2(|B^*|)+f_3^2(|C|)\right\|^\frac{1}{2},\left\|g_4^2(|D^*|)\right\|^\frac{1}{2}\right\}
 \\&&+\max\left\{\left\|g_1^2(|A^*|)\right\|^\frac{1}{2},\left\|f_2^2(|B|)+g_3^2(|C^*|)+f_4^2(|D|)\right\|^\frac{1}{2}\right\}\\&=&
 \max\left\{\left\|f_1^2(|A|)+g_2^2(|B^*|)+f_3^2(|C|)\right\|^\frac{1}{2},\left\|g_4(|D^*|)\right\|\right\}
 \\&&+\max\left\{\left\|g_1(|A^*|)\right\|,\left\|f_2^2(|B|)+g_3^2(|C^*|)+f_4^2(|D|)\right\|^\frac{1}{2}\right\}.
\end{eqnarray*}%
Taking the supremum over all unit vectors $\mathbf{x}$ we get the desired result.
\end{proof}
\begin{corollary}
 Let $A, B, C, D\in {\mathbb B}({\mathscr H})$. Then
\begin{eqnarray*}
\omega \left(\left[\begin{array}{cc}
 A&B\\
 C&D
 \end{array}\right]\right)&\leq&\max\left\{\left\||A|^{2\alpha}+|B^*|^{2\gamma}+|C|^{2\mu}\right\|^\frac{1}{2},\left\||D^*|^{\omega}\right\|\right\}
 \\&&+\max\left\{\left\||A^*|^{\beta}\right\|,\left\||B|^{2\zeta}+|C^*|^{2\nu}+|D|^{\kappa}\right\|\right\},
\end{eqnarray*}%
where $\alpha+\beta=\gamma+\zeta=\mu+\nu=\omega+\kappa=1$. In particular, \begin{eqnarray*}
\omega \left(\left[\begin{array}{cc}
 A&B\\
 0&0
 \end{array}\right]\right)\leq\max\left\{\left\||A^*|^{\beta}\right\|,\left\||B|^{\zeta}\right\|\right\}+\left\||A|^{2\alpha}+|B^*|^{2\gamma}\right\|^\frac{1}{2},
 \end{eqnarray*}%
in which $\alpha+\beta=\gamma+\zeta=1$.
 \end{corollary}
\bigskip
%===================================================================================================================================
\bibliographystyle{amsplain}

\end{document}